\documentclass[11pt,twoside]{article}
\usepackage{latexsym}
\usepackage{amssymb,amsbsy,amsmath,amsfonts,amssymb,amscd}
\usepackage{graphicx,color}
\usepackage{float,url}
\usepackage{cancel}
\usepackage{mathtools}
\usepackage[colorlinks,linkcolor=red,citecolor=blue]{hyperref}

\newcommand \blue[1] {\textcolor{blue}{#1}}

\newcommand{\commentout}[1]{}
\newcommand{\R}{\mathbb{R}}

\newcommand {\al} {\alpha}
\newcommand {\eps}  {\varepsilon}

\newcommand {\sg} {\sigma}

\newcommand {\Chi} {{\bf \raise 2pt \hbox{$\chi$}} }
\newcommand {\f}   {\frac}
\newcommand {\p}   {\partial}
\newcommand{\dis}{\displaystyle}

\newtheorem{theorem}{Theorem}
\newtheorem{lemma}[theorem]{Lemma}

\newtheorem{remark}[theorem]{Remark}
\newtheorem{proposition}[theorem]{Proposition}

\newcommand{\qed}{{ \hfill
                       {\unskip\kern 6pt\penalty 500 \raise -2pt\hbox{\vrule\vbox to 6pt{\hrule width 6pt
                       \vfill\hrule}\vrule} \par}   }}
\title{Deriving sub-diffusion equations}

\author{
Beno\^ \i t Perthame\thanks{Sorbonne Universit{\'e}, CNRS, Universit\'{e} de Paris, Inria, Laboratoire Jacques-Louis Lions UMR7598, F-75005 Paris. 
Email : Benoit.Perthame@sorbonne-universite.fr}
\and Min Tang\thanks{ School of Mathematical Sciences and Institute of Natural Sciences, MOE-LSC, CMA-Shanghai, Shanghai Jiao Tong University, China. Email: tangmin@sjtu.edu.cn}
}
\date{\today}

\begin{document}
\maketitle
\pagestyle{plain}
\pagenumbering{arabic}

\begin{abstract} 
Sub-diffusion equations are used in a large range of applications including fluids, plasma physics and biology. Their mathematical analysis is advanced even if a much larger literature addresses super-diffusions. 
The goal of this paper is to provide the microscopic mechanism and rigorous derivation of sub-diffusions when the waiting time distribution of particles follows an age-structured equation and jumps occur at each renewal. The major difficulty to recover sub-diffusions, unlike normal diffusions, is that the assumption of long waiting time implies lack of integrability for the age equilibrium. This prevents to establish strong a priori estimates. Here, the Laplace transform plays the role that Fourier transform plays for the more traditional case of fast diffusions.
\end{abstract} 
\vskip .7cm

\noindent{\makebox[1in]\hrulefill}\newline
\section{Introduction}
\label{sec:intro}

Sub-diffusion, characterized by a mean squared displacement (MSD) scaling as $\langle x^2(t) \rangle \propto t^\alpha$ with $0 < \alpha < 1$, has been observed across diverse physical and biological systems \cite{METZLER20001,sub3,sub4,sub7}. In living cells, sub-diffusive transport is prevalent: lipid granules in fission yeast exhibit anomalous diffusion in the cytoplasm \cite{sub4}, insulin granules in MIN6 cells show heavy-tailed waiting time distributions during microtubule binding/unbinding \cite{sub8}, and hydration water on protein surfaces undergoes sub-diffusion due to trapping in heterogeneous potential landscapes \cite{Tan2018, Li2022}. For a comprehensive review of sub-diffusion in natural systems, see \cite{Meroz2015}.

Mathematically, sub-diffusion is often modeled using fractional calculus, introducing nonlocal time derivatives to account for memory effects. The Caputo \textit{Riemann–Liouville} fractional derivative, defined as  
\[
{}_{0}^{RL}D_{t}^{\alpha}f(t) =\partial_t\int_0^t\frac{f(\tau)}{(t-\tau)^\al } d\tau,
\]
naturally arises in sub-diffusion equations such as  
\[
{}_{0}^{RL}D_t^\alpha u - \text{div}(\kappa(u) \nabla u) = f,
\]  
where $\alpha$ governs the degree of sub-diffusion \cite{Caputo1999}. This formalism captures memory through the power-law kernel $(t - \tau)^{-\alpha}$.

A prevalent microscopic model for sub-diffusion is the continuous-time random walk (CTRW), where particles alternate between jumps and trapping events with stochastic dwell times. When the trapping time distribution follows a power law decay $\psi(t) \sim t^{-(1+\alpha)}$ ($0 < \alpha < 1$) for $t\gg 1$, the ensemble-averaged MSD exhibits sublinear scaling $\langle x^2(t) \rangle \propto t^\alpha$. This non-ergodic behavior arises because longer observations uncover deeper traps, leading to a breakdown of ergodicity \cite{METZLER20001}.

Prior works have linked CTRW models to sub-diffusion equations. For instance, \cite{PhysRevE.61.132} showed equivalence between their Green's functions, while \cite{SIAP2015} derived the fractional Fokker-Planck equation from a generalized master equation for CTRW, connecting power-law trapping kernels to fractional time derivatives. In CTRW models, the power-law trapping distribution is typically assumed, which implies that the trapped particles are always in equilibrium. However, a fundamental question remains: \textit{What microscopic mechanisms generate power-law trapping distributions?}

In this paper, we address this gap by proposing an \textbf{age-structured model} for particle trapping dynamics. Age-structured models, widely used in demography, ecology, epidemiology \cite{thieme,DGMS_book,DiH}, track populations partitioned by age $a$ (here, trap residence time). We model the probability density $u(a, t)$ of particles trapped for time $a$ at time $t$ using the age structured (McKendrick-Von Foerster) equation~\cite{IannelliBook}:  
\begin{equation}\label{intro:age}
\frac{\partial u(a, t)}{\partial t} + \frac{\partial u(a, t)}{\partial a} = -d(a)u(a, t),
\end{equation}
where $d(a)$ is the age-dependent escape rate. When $d(a) = \frac{\alpha}{1 + a}$ ($0 < \alpha < 1$), the escape probability decreases with trap residence time, leading to a power-law waiting time distribution $\Psi(a) \sim a^{\alpha - 1}$ for large values of~$a$. 

By coupling this age-structured trapping model with spatial jumps, we rigorously derive the sub-diffusion equation and identify scaling conditions under which the microscopic and macroscopic descriptions coincide. Notably, we establish a correspondence between the initial conditions of the age-structured model (e.g., initial trap residence times) and those of the fractional sub-diffusion equation, addressing a critical gap in the literature. This connection provides mechanistic insights into sub-diffusive behavior in complex systems. An age-structured kinetic model was used in \cite{EstradaPainter_2019} to derive a space-time fractional diffusion equation, but the derivation is formal, and the initial condition was not identified.

On the one hand, the mathematical literature on problems involving \textbf{Caputo derivatives} and partial differential equations is extensive, spanning topics from the porous medium equation~\cite{Allen2017} to mathematical biology and medicine~\cite{BerryLG2016,FEB_2024} (see also the references therein). Sobolev regularity in $L^p$ spaces has also been explored (cf.~\cite{DongKim_2020}), quasi-linear problems have been investigated~\cite{zacher_2012}, and Besov regularity has been established. However, these works do not address the derivation of such equations from finer-scale mechanisms.  
On the other hand, equations combining age-dependent escape rates with kinetic equations have been derived from particle scattering studies~\cite{CagGolse_2010, MS_2010}. Various limiting models have also been investigated: for instance, the authors of~\cite{GoudonFrank_2010} and~\cite{FrankSun_2018} analyze the normal diffusion limit and super-diffusion limit, respectively. A consistent assumption in their work is that the equilibrium in Eq.~\eqref{intro:age}—specifically, $u_{\rm eq}= e^{\int_0^a d(a')da' }$—is integrable. In our analysis, a major difficulty is that the sub-diffusion regime emerges precisely when this equilibrium is non-integrable. This corresponds to long waiting times between jumps and also hinders the establishment of strong a priori estimates.  Finally, to our knowledge, this is the first rigorous derivation of a sub-diffusion equation from a microscopic model via asymptotic analysis, with appropriate initial conditions explicitly identified. This regime is examined in~\cite{CGA_2019} under a different scaling, yielding a Hamilton-Jacobi equation, while self-similar solutions are constructed in~\cite{BerryLG2016}. Our approach differs from these works, and—analogous to the use of Fourier transforms for super-diffusion—we employ Laplace transforms to derive the sub-diffusion equation.

The organization of this paper is as follows: Section~\ref{sec:LaplaceTime} introduces the notations and demonstrates how to apply the Laplace transform to prove the decay rate for large time. Section~\ref{sec:agetime} establishes the power-law decay rate for an age-structured model in the large-time regime, while Section~\ref{sec:sub-diffAge} derives the sub-diffusion equation when spatial jumps are modeled by a nonlocal kernel. The corresponding scales of the age-structured jump model are identified, and the contribution from the initial condition is shown to emerge as a source term in the sub-diffusion equation.

\section{Laplace transform and time decay}
\label{sec:LaplaceTime}

As a preparation to the case with space, we introduce notations and material concerning the Laplace transform and its use for proving large time decay rates.
We consider functions $u(t)$, $t\in \R$ satisfying $u(t)=0$ for $t\leq 0$ and $u(t) \geq 0$ for $t>0$. The Laplace transform of $u(t)$ is denoted by $\widehat u(s)$, 
\[
 \widehat u (s) ={\cal L}_s u=\int_0^\infty e^{-st} u(t) dt, \qquad s \geq 0.
\]
Notice that, thanks to our convention for $t\leq 0$, 
 $
 s \;  \widehat u (s) =\int_0^\infty e^{-st} \frac d {dt} u(t) dt,
 $
 which can be rewritten as
 \begin{equation}\label{transitst}
 s{\cal L}_s u={\cal L}_s \p_t u.
 \end{equation}
We also remark that, for $u\geq 0$, 
 \[
 u\in L^1(\R)  \qquad \Longleftrightarrow \qquad \widehat u \in L^\infty(0,\infty).
 \]
 Moreover,\[
 u(t) \approx t^{\alpha-1} , \; t\gg 1, \quad 0< \al <1  \qquad \text{implies} \qquad   \widehat u (s) \approx s^{-\al} \to \infty \quad s \approx 0,
 \]
 because, with the change of variable $t \mapsto \sg=st$, 
 \begin{align}\label{salphat}
 \int_0^\infty t^{\al-1} e^{-st} dt =s^{-\al }   \int_0^\infty \sg^{\al-1} e^{-\sg} d\sg. 
 \end{align}
 Here since $\alpha\in(0,1)$, $\sigma^{\alpha-1}e^{-\sigma}$ is integrable on $[0,\infty)$.
 The reverse implication is not so simple and can only be interpreted in an integral sense as given by Lemma ~\ref{lem:integral}.
\begin{lemma}\label{lem:integral}
Assume that, $u(t)\geq 0$  for $t>0$, and that for some $\alpha$ with $0\leq \al \leq 1$ and  constants $\underline K$, $\overline K$, the following holds for $0<s \leq 1$: $\underline K s^{-\al} \leq  \widehat u (s) \leq \overline K s^{-\al }$ . Then for all $\eps >0$ and for some constants $\underline K_1(\al)$ and $\overline K_1(\al)$ that are independent of $\eps$,
\begin{equation} \label{LT:timedecay}
 \frac {\underline K_1(\al)} \eps \leq \int_{0}^\infty u(t) t^{-\al- \eps } dt , \qquad   \int_{t\geq 1} u(t) t^{-\al-\eps } dt \leq  \frac { \overline K_1(\al)}{ \eps}.
\end{equation}
\end{lemma}

\begin{proof}
We prove the second inequality. We first use the immediate control
 \[
 \int_0^1 s^{-1+\al +\eps } \int_0^\infty  e^{-st} u(t) dt ds  =  \int_0^1 s^{-1+\al +\eps } \widehat u (s)   ds \leq  \overline K \int_0^1 s^{-1+\eps} ds =  \frac { \overline K} \eps,
 \]
and thus we conclude, with the change of variable $s \mapsto \sg= st$, that
 \begin{align}\label{uts}
 \int_1^\infty u(t)  \int_0^1 s^{-1+\al+\eps }  e^{-st} ds dt  = \int_1^\infty u(t) t^{-\al- \eps }  \int_0^t \sg^{-1+\al +\eps }  e^{-\sg} d\sg dt \leq  \frac { \overline K} \eps.
 \end{align}
 This gives the upper bound because $t \geq 1$ and $ \int_0^1 \sg^{-1+\al +\eps }  e^{-\sg} d\sg \geq \frac{e^{-1}}{1-\al -\eps}>\f{e^{-1}}{1-\alpha} $.
 
 For the lower bound, the same calculations give, on the one hand 
 \[
 \int_0^1 s^{-1+\al +\eps } \int_0^\infty  e^{-st} u(t) dt ds \geq  \frac { \underline K} \eps,
 \]
 and on the other hand, similar equality as in \eqref{uts} gives 

 \[
 \frac {\underline K} \eps \leq \int_0^\infty u(t) t^{-\al -\eps }  \int_0^t \sg^{-1+\al +\eps }  e^{-\sg} d\sg dt  \leq \int_0^\infty u(t) t^{-\al -\eps }  dt \int_0^\infty \sg^{-1+\al +\eps }  e^{-\sg} d\sg .
 \] Then $\underline K_1(\al)$ is given by $\underline{K}/ \int_0^\infty \sg^{-1+\al +\eps }  e^{-\sg} d\sg$.
 \end{proof}

\begin{remark} The conclusion \eqref{LT:timedecay} is a statement compatible with the large time decay of $u(t)$. Namely, when $u(t) \approx t^{\alpha-1}$,  \eqref{LT:timedecay} is satisfied. Notice also that the same conclusion follows with the weaker assumption for $s$ such that
\[\frac { \underline K} \eps \leq  \int_0^1 s^{-1+\al +\eps } \widehat u (s)   ds \leq  \frac { \overline K} \eps.\]
\end{remark}

\section{Age structured equation and time decay}
\label{sec:agetime}
 
To apply the above considerations about the Laplace transform, we consider the age structured equation:
\begin{align} \label{eq:hom} \begin{cases}
 \partial_t u(t,a) +\partial_a  u(t,a) + d(a)  u(t,a) =0, \quad a\geq 0, \; t \geq 0,
 \\
 U(t):=u(t,0)= \int_0^\infty d(a)  u(t,a) da, 
 \\
 u(0,a)=u^0(a),
\end{cases} \end{align}
and $u(t,a)$ is extended to $0$ for $t<0$. 
Here $d(a)>0$ and the initial data is supposed to satisfy
\begin{align} \label{as:homID}
 u^0(a)\geq 0, \quad  u^0(a) \in L^1(0,\infty), \quad M:=\int_0^\infty  u^0(a) da<\infty.
\end{align}
The problem is conservative, i.e., 
\[
\int_0^\infty u(t,a) da = M, \qquad \forall t \geq 0.
\]
 We define
 \begin{align} \label{def:D}
 D(a)= \int_0^a d(a') da'.
\end{align}
When $e^{-D}$ is integrable, it is standard that solutions of \eqref{eq:hom} converge to a multiple of $e^{-D}$ \cite{IannelliBook,PerthameTEB}. When it is not, it remains that $de^{-D}$ is integrable. For this reason, we investigate the equation satisfied by the renewal term $U(t)$.

\paragraph{Useful formulas.}
In order to avoid abstract technicalities, we often use  
 \begin{align} \label{as:hom}
 d(a)= \frac{\alpha}{1+a}, \qquad e^{-D(a)} =  \frac{1}{(1+a)^\alpha}, \qquad \alpha\in(0,1).
 \end{align} 
Because $\alpha$ is fixed, we set 
\begin{align}
\Gamma_\alpha:=\Gamma(1-\al) = \int_0^\infty  \frac{e^{-\sg} }{\sg^\al} d\sg \label{Gamma}.
\end{align} 
For any given function $P(a)\in L^1[0,\infty)$, the following formula is useful
\begin{equation} \label{FundLaplace}
 {\cal L}_s  \partial_t \int_0^t \frac{P(a)}{(t-a)^\al}d\tau = \Gamma_\al s^{\al} {\cal L}_s  P.
\end{equation}
This formula holds because, integrating by parts in $t$, we can compute
\begin{align*}
\int_0^\infty e^{-st} \partial_t \int_0^t \frac{P(a)}{(t-a)^\al}da dt &= s \int_0^\infty e^{-st} \int_0^t \frac{P(a)}{(t-a)^\al}da dt 
=s \int_0^\infty P(a) e^{-a s}\int_{a}^\infty \frac{e^{-s(t-a)} }{(t-a)^\al}dt da
\\[10pt]
&=s \int_0^\infty P(a) e^{-a s}\int_{0}^\infty \frac{e^{-st} }{ t^\al}dt da
=s^\al \int_0^\infty P(a) e^{-a s}da \int_{0}^\infty \frac{e^{-\sg} }{\sg^\al}d\sg, 
\end{align*}
where the last equality follows from the change of variable $t \mapsto \sg= st$.

In this paper, we also use a related formula
\begin{equation} \label{FundLaplace2}
 {\cal L}_s  \partial_t \int_0^t \frac{P(\tau)}{(1+t-\tau)^\al}d\tau = s \int_0^\infty \frac{e^{-sa}}{(1+a)^\al} da\; {\cal L}_s  P .
\end{equation}
It is derived exactly as for \eqref{FundLaplace}.

\paragraph{The equation for $U(t)$.}

 Applying the Laplace transform to the age structured model \eqref{eq:hom}, we find
\begin{align*}
 \partial_a  \widehat u(s,a) +(s+d(a)) \widehat u(s,a) = u^0(a).
\end{align*}
The solution is given by 
\begin{align*}
 \widehat  u(s,a) = \widehat  u(s,0) e^{-D(a) -sa} + e^{-D(a) -sa} \int_0^a  u^0(a') e^{D(a') +sa'} da'.
\end{align*}
Thus, the Laplace transform of the renewal term  $\widehat  U(s) := \widehat u(s,0)$ is determined by
\begin{align}\label{hatU}
 \widehat  U(s) =  \widehat  U(s)  \int_0^\infty d(a) e^{-D(a) -sa} da + \int_0^\infty d(a) e^{-D(a) -sa} \int_0^a  u^0(a') e^{D(a') +sa'} da' da.
\end{align}
To treat the first term on the right hand side (RHS) of \eqref{hatU},  we compute
\begin{align*}
 \int_0^\infty d(a) e^{-D(a) -sa} da =- s \int_0^\infty e^{-D(a) -sa} da +  \int_0^\infty (s+d(a)) e^{-D(a) -sa} da =- s \int_0^\infty e^{-D(a) -sa} da+1.
\end{align*}
Therefore, \eqref{hatU} becomes, using the Fubini theorem and the same argument as before but replacing the lower bound of the integrations from $0$ to $a'$,
\begin{align*}
s \int_0^\infty e^{-D(a) -sa} da \;   \widehat  U(s)& = \int_0^\infty  u^0(a') e^{D(a') +sa'}  \int_{a'}^\infty d(a) e^{-D(a) -sa}  da da' \notag
\\
& = \int_0^\infty  u^0(a)da- s \int_0^\infty  u^0(a') e^{D(a') +sa'}  \int_{a'}^\infty e^{-D(a) -sa}  da da' .
\end{align*}
When $d(a)$ is given by \eqref{as:hom}, the above equation can be written as 
\begin{align} 
s \int_0^\infty \frac{e^{-sa}}{(1+a)^\alpha} da \;   \widehat  U(s)
& =\int_0^\infty  u^0(a)da- s \iint_0^\infty  u^0(a')\frac{(1+a')^\alpha}{(1+t+a')^\alpha} e^{-st} dt da'.\label{LT:decay0}
\end{align}
Here the last equality is obtained by the change of variable $t=a-a'$.

Using the formula \eqref{transitst} and \eqref{FundLaplace2}, we may rewrite \eqref{LT:decay0} as
\begin{align*} 
 {\cal L}_s  \partial_t \int_0^t \frac{U(\tau)}{(1+t-\tau)^\al}d\tau 
 &=\int_0^\infty  u^0(a)da -s {\cal L}_s \int_0^\infty u^0(a) \frac{(1+a)^\alpha}{(1+t+a)^\alpha}da \notag\\
 &=\int_0^\infty  u^0(a)da - {\cal L}_s  \partial_t \int_0^\infty u^0(a) \frac{(1+a)^\alpha}{(1+t+a)^\alpha}da,
\end{align*}
and thus
\begin{align} \label{LT:decayderivative}
\partial_t \int_0^t \frac{U(\tau)}{(1+t-\tau)^\al}d\tau = \delta(t) \int_0^\infty  u^0(a)da -\partial_t \int_0^\infty u^0(a) \frac{(1+a)^\alpha}{(1+t+a)^\alpha}da.
\end{align}

We finally obtain the equation that $U(t)$ satisfies, which is in the following integral form:
\begin{align} \label{LT:decay1}
\int_0^t \frac{U(\tau)}{(1+t-\tau)^\al}d\tau =\int_0^\infty  u^0(a)da - \int_0^\infty u^0(a) \frac{(1+a)^\alpha}{(1+t+a)^\alpha}da.
\end{align}
For $t>0$, differentiating both sides of \eqref{LT:decay1} with respect to 
$t$ yields \eqref{LT:decayderivative}, whereas both sides of \eqref{LT:decay1} vanish at $t=0$. The integral on the left-hand side of \eqref{LT:decay1} has no singularity in the integration kernel; hence, it is not a Caputo derivative but still incorporates the memory effect.

From the explicit exact formula~\eqref{LT:decay1}, we have the time decay property of $U(t)$:
\begin{proposition}[Time decay for age structured eq.] \label{prop:asdecay} 
With assumptions  \eqref{as:homID} and \eqref{as:hom} with $\alpha \in (0,1)$,  the solution of \eqref{eq:hom} decays as  $u(t,0)=U(t) \approx t^{\al-1} $ for $t\gg 1$,  in the sense that
\begin{align} \label{asympt_Simple}
\lim_{t \to \infty} \int_0^t \frac{U(\tau)}{(1+t-\tau)^\al}d\tau =\int_0^\infty u^0(a) da.
\end{align}
In particular, for all $0< \delta <1$, we have for some constants $C_\pm(\alpha, u^0)$,
\begin{align} \label{asympt_Simple2}
\frac{C_-(\alpha,u^0)}{\delta} \leq \int_0^\infty t^{-\al-\delta} U(t) dt, \qquad \int_1^\infty t^{-\al-\delta} U(t) dt \leq M \frac{C_+(\alpha)}{\delta} .
\end{align}
\end{proposition}

\begin{proof}
From \eqref{LT:decay1}, we obtain \eqref{asympt_Simple} since
\[
\int_0^t \frac{U(\tau)}{(1+t-\tau)^\al}d\tau = \int_0^\infty  u^0(a)da +o(1) , \qquad t\gg 1.
\]

The inequality \eqref{asympt_Simple2} follows by integrating \eqref{LT:decay1} with the weight $t^{-(1+\delta)}$ for $t>1$. For the upper bound, due to \eqref{LT:decay1} and the positivity of $u^0(a)$ in \eqref{as:homID}, we argue as follows:  
\[
\int_1^\infty t^{-(1+\delta)}\int_{0}^t \frac{U(\tau)}{(1+t-\tau)^\al}d\tau dt <\int_1^\infty t^{-(1+\delta)}dt\int_0^\infty u^0(a)da\leq \frac 1 \delta \int_0^\infty  u^0(a)da.
\]
Using the Fubini theorem, we find with a further reduction,
\[
\int_1^\infty U(\tau) \int_{\tau}^\infty \frac{t^{-(1+\delta)}}{(1+t-\tau)^\al}dt d\tau \leq \frac 1 \delta \int_0^\infty  u^0(a)da.
\]
from which we conclude the upper bound by using that, when $\tau\geq 1$,
\[\frac{t^{-(1+\delta)}}{(1+t-\tau)^{\alpha}}\geq t^{-1-\alpha-\delta}.\] We may also write 
\[
 \int_1^\infty t^{-(1+\delta)}\int_{0}^t \frac{U(\tau)}{(1+t-\tau)^\al}d\tau dt=\int_0^\infty U(\tau) \int_{\tau}^\infty \frac{t^{-(1+\delta)}}{(1+t-\tau)^\al}dt d\tau .
\]
From \eqref{LT:decay1}, the left hand side of the above equation can be bounded from below such that 
\[
\int_1^\infty t^{-(1+\delta)}\Big[\int_0^\infty u^0(a)da-\int_0^\infty u^0(a)\frac{(1+a)^\al}{(2+a)^\al}da\Big]dt \leq \int_1^\infty t^{-(1+\delta)}\int_{0}^t \frac{U(\tau)}{(1+t-\tau)^\al}d\tau dt, 
\]
from which we conclude the lower bound in \eqref{asympt_Simple2}.
\end{proof}

\begin{remark}
The interpretation is that the density, which is conserved, is spreading to infinity as for the heat equation in the whole space. The factor $d(a)$ gives less weight for large $a$ and thus does not see any mass.
\end{remark}
\section{sub-diffusion equation}
\label{sec:sub-diffAge}

\subsection{Assumptions and main result}

Next, we consider the scenario where particles undergo spatial jumps upon renewal while inside a trap located at position $y$, the duration time of particles follows an age-structured model. Upon exiting the trap, the transition probability of jumping to a new position $x$ is governed by a rate $\omega_\epsilon(x,y)$, which satisfies the following conditions: it is a measurable function such that
\begin{equation}\label{as:Subomega}
 \omega_\epsilon(x,y) \geq 0, \qquad \int_{\mathbb{R}^d} \omega_\epsilon(x,y) \, dy = 1, \qquad \text{and} \qquad \omega_\epsilon(x,y) = \omega_\epsilon(y,x).
\end{equation}
Consequently, the age-structured jump model can be formulated as:
 \begin{align} \label{eq:AgeSpace} 
 \begin{cases}
 \eps^{\beta}\partial_t u_\eps(t,a,x) +\partial_a  u_\eps(t,a,x) + d(a)  u_\eps(t,a,x) =0 , \quad a\geq 0, \; x \in \R^d, \; t \geq 0,
 \\[5pt]
\dis U_\eps(t,x):= u_\eps(t,0,x)= \int_0^\infty \int_{\R^d} d(a)  u_\eps(t,a,y) \omega_\eps (x,y)dady, 
 \\[5pt]
 u_\eps(0,a,x)=u^0(a,x).
 \end{cases} \end{align}
 As before, we define 
 \[
 D(a) =\int_0^a d(a')da' \qquad \big(e^{-D(a)}= \frac{1}{(1+a)^\alpha} \text{ in case of \eqref{as:hom}}\big) .
 \]
 We assume that the initial data satisfies, for some constant $C^0$,
 \begin{align} \label{as:SubID}
 u^0\geq 0, \qquad  e^{D(a)} u^0 \in L^1\big((0,\infty)\times \R^n\big), \qquad  u^0 (a,x) \leq C^0e^{-D(a)}.
 \end{align}
The parameter $\eps$ is defined from the scaled probability that a particles at position $y$ will jump to position $x$. We assume that, additionally to~\eqref{as:Subomega}, the transition probability $\omega_\eps(x,y)$ satisfies,  as $\eps \to 0$, 
\begin{align} \label{as:Subomegaeps}
\begin{cases}
\dis  \int_{\R^d} \frac{1}{\eps^2}  (x-y)\otimes(x-y) \omega_\eps(x,y) dy \to A(x) >0  \qquad \text{a symmetric positive matrix},
\\[10pt]
\dis \frac{1}{\eps^3} \int_{\R^d}  |x-y|^3 \omega_\epsilon (x,y)dy \in L^\infty(\R^d) \qquad  \text{ uniformly in } \eps.
 \end{cases}
 \end{align}
 In one dimension, a simple example of $\omega_\epsilon$ that satisfies the above requirements is 
 \begin{equation*}\label{eq:omegaexample}
\omega_\eps(x,y)=\frac{1}{\eps}f_{\eps}(\frac{|x-y|}{\eps})=\frac{1}{\eps}f_{\eps}(|z|)=\left\{
\begin{array}{lll}
\frac{1-|z|}{\eps}, &z\in\left[-1,1\right), 
\\[5pt]
0, &\mbox{$z\in(-\infty,-1)\cup(1,+\infty)$}
\end{array}\right.
\end{equation*}
 The value of $\beta$ in Eq.~\eqref{eq:AgeSpace} depends on the jump rate $d(a)$ and is  determined further.
 \\
 
 With these assumptions, the problem is conservative and we fix the total mass $M$ defined as
 \[
 \int_0^\infty \int_{\R^d} u_\eps (t,a,x)  da dx=  \int_0^\infty \int_{\R^d} u^0(a,x)  da dx=:M, \qquad \forall t\geq 0 .
 \]

Since $e^{-D(a)}$ is a steady state solution, the comparison principle shows that the solution satisfies, 
 \[
 u_\eps (t,a,x) \leq C^0 e^{-D(a)}, \qquad \qquad U_\eps (t,x) \leq C^0, \qquad \quad \forall t\geq 0 .
 \]
Also, with further assumptions on the initial data, further regularity is available, see subsection~\ref{sec:tDecay}.

Let us set 
\begin{align*} 
\rho_\eps(t,x)=\int_0^\infty u_\eps(t,a,x)da, \qquad \qquad  \rho^0(x)= \int_0^\infty u^0(a,x) da
\end{align*}
and to simplify notations, 
\[
\nabla \nabla(A\rho):= \sum_{i,j} \partial^2_{i,j} [A_{i,j} \rho].
\]
We consider the case when $d(a)$ is determined by \eqref{as:hom}, i.e., $D(a)=\frac{1}{(1+a)^\alpha}$, or when $d(a)=D_0$, i.e., $D(a)e^{D_0 a}$ (this is simpler than taking $\alpha>1$ and has the same effect). 
Consider the weak limit, in bounded measures, $\rho_0$ of $\rho_\eps$, as $\eps \to 0$ (a priori extraction of a subsequence is needed but uniqueness of the limit shows that the full family converges). The parameter $\alpha$ will decide of the time scale $\beta$ and of the limiting equation, namely, between normal diffusion when $d(a)=D_0$: 
\begin{align} \label{NormaDiff}
\begin{cases}
\partial_t \rho_0(t,x) - D_0\nabla \nabla[A(x) \rho_0]=0,
\\[5pt]
\rho_0(0,x)= \rho^0(x),
\end{cases}
\end{align} 
and sub-diffusion when $d(a)$ is determined by \eqref{as:hom} with $0< \alpha <1$:
\begin{align} \label{AbNormaDiff}
\begin{cases}
\partial_t \int_0^t \frac{\rho_0(\tau,x)}{(t-\tau)^\alpha} d\tau- \nabla \nabla [A(x) \rho_0(t,x)]= t^{-\al}\rho^0(x),
\\[5pt]
\rho_0(0,x)= 0 \qquad \text{for} \quad t< 0.
\end{cases}
\end{align} 

\begin{remark}\label{remark5}
It is interesting to note that the initial contribution of $u^0(a,x)$ appears solely in the source term rather than in the initial condition, which differs from the classical normal diffusion case. When $t\to 0+$,  the dominant term on the left-hand side of \eqref{AbNormaDiff} becomes $t^{-\alpha}\rho_0(0,x)$. Comparing with the right-hand side of \eqref{AbNormaDiff}, this indicates that $\rho_0(0+,x)=\rho^0(x)$, meaning $\rho_0(t,x)$ jumps from 0 to $\rho^0(x)$ at $t=0$.
\end{remark}

For the sub-diffusion case, we also have, when $\eps\to 0$, 
\[
\int_0^t \frac{U_0(\tau,x)}{(t-\tau)^\alpha} d\tau=U_0(0,x) O(t^{1-\alpha} ) \quad \text{for} \quad t\approx 0.
\]

These equations are understood in the weak form, testing against a smooth test function $\psi(t,x)$ vanishing at $t=T$ and compactly supported in $\R^n$. For Eq.~\eqref{NormaDiff} this is
\[
\int_0^T \int_{\R^n}  \rho_0(t,x) \big[-\partial_t \psi(t,x) - D_0 \sum_{i,j}  A_{i,j}(x) \partial^2_{i,j} \psi(t,x)\big]dx dt= \int_{\R^n} \psi(0,x) \rho_0(0,x) dx .
\]
And for the sub-diffusion case, it is
\[
\int_0^T \int_{\R^n} \rho_0(t,x) \big[- \int_t^T \frac{\partial_\tau \psi(\tau,x)}{(\tau-t)^\alpha}d\tau - \sum_{i,j}  A_{i,j}(x) \partial^2_{i,j} \psi(t,x)\big]dxdt= \int_{t=0}^T \int_{a=0}^\infty \int_{\R^n}t^{-\alpha} \psi(t,x) u^0(a,x) dx da dt .
\]

\begin{theorem} [Derivation of the macroscopic limit] \label{th:limit} 
Assume  \eqref{as:SubID}, \eqref{as:Subomega}, \eqref{as:Subomegaeps} and choose $d(a)$ as in~\eqref{as:hom}. Then, we have
\\[2pt]
(i) \, for $\al > 1$ and $\beta=2$, the limit of $\rho_\eps(t,x)$ satisfies the diffusion equation~\eqref{NormaDiff},
\\[5pt]
(ii) for $\al \in(0, 1) $ and $\alpha \beta=2$, the limit of $\rho_\eps(t,x)$ satisfies the sub-diffusion equation~\eqref{AbNormaDiff}.
\end{theorem}
\subsection{Proof of the main theorem}
\begin{proof}
We argue with the Laplace transform of Eq.~\eqref{eq:AgeSpace}, which gives
\begin{align*} 
 \partial_a  \widehat u_\eps(s,a,x) + [\eps^{\beta} s+ d(a) ]  \widehat u_\eps(s,a,x) =  \eps^{\beta} u^0(a,x),
\end{align*}
With $\widehat U_\eps(s,x)= \widehat u_\eps (s,0,x)$ and the definition of $D(a)$, the solution is given by
\begin{align}\label{usax} 
\widehat u_\eps(s,a,x)  = \widehat U_\eps(s,x) e^{-D(a)  -\eps^{\beta} sa } 
+ \eps^{\beta} e^{-D(a)-\eps^{\beta} sa }   \int_0^a u^0(a',x) e^{D(a') +\eps^{\beta} sa' }  da'.
\end{align}
From this representation formula, we may deduce two macrocospic identities. Firstly, integrating in~$a$, we find 
\begin{align} \label{eq1:rho}
\widehat \rho_\eps(s,x)  = \widehat U_\eps(s,x) \int_0^\infty e^{-D(a)  -\eps^{\beta} sa } da
+ \eps^{\beta} \text{ID}^{(1)}_\eps(s,x)
\end{align}
where, changing $a$ to $a'+t$, we have
\begin{align}
\text{ID}^{(1)}_\eps(s,x)&=\int_0^\infty e^{-D(a)-\eps^{\beta} sa }   \int_0^a u^0(a',x) e^{D(a') +\eps^{\beta} sa' }  da'da \notag\\
&=\int_0^\infty u^0(a’,x)\int_{0}^\infty e^{D(a’)-D(a'+t)-\eps^\beta st} dtda'. \label{eq:ID1}
\end{align}
Secondly, using the definition of $U_\eps(t,x)=u_\eps(t,0,x)$ in \eqref{eq:AgeSpace}, we multiply both sides of~\eqref{usax} (written in the variable $y$ rather than $x$) by $d(a) \omega_\eps (x,y)$, integrate with respect to $a$ from $0$ to $\infty$ and $y$ in $\R^d$. This gives
\begin{align}
\widehat U_\eps(s,x) &= \int_0^\infty \!\!\int_{\R^n} d(a) e^{-D(a)  -\eps^{\beta} sa } \widehat U_\eps(s,y) \omega_\eps (x,y)dady + \eps^\beta \text{ID}^{(2)}_\eps(s,x)  \notag\\
&= \int_{\R^n}   \widehat U_\eps(s,y)\left[ 1- \eps^{\beta} s \int_0^\infty  e^{-D(a)  -\eps^{\beta} sa }da \right] \omega_\eps (x,y) dy + \eps^\beta\text{ID}^{(2)}_\eps(s,x)  , \label{Uepsspace} 
\end{align}
where the initial data term is given by 
\begin{align} 
 \text{ID}^{(2)}_\eps(s,x)&
=  \int_0^\infty \!\! \int_{\R^n} d(a) e^{-D(a)  -\eps^{\beta} sa }   \int_0^a u^0(a',y) e^{D(a')+ \eps^{\beta} sa' }  da' \; \omega_\eps (x,y)dady \notag
 \\
&=  \int_{a'=0}^\infty \int_{\R^n} u^0(a',y) e^{D(a') + \eps^\beta sa' } \omega_\eps (x,y)
 \int_{a'}^\infty d(a)e^{-D(a)  -\eps^{\beta} sa }     da da'dy \notag
\\
&=   \int_{a'=0}^\infty \int_{\R^n} u^0(a',y) \omega_\eps (x,y)
\big[1- \eps^\beta s 
 \int_{a'}^\infty e^{-D(a) +D(a') -\eps^{\beta} s(a-a') }  \big]   da da'dy.\label{IDeps}
\end{align}

Our next step is to eliminate $\widehat U_\eps(s,x)$ from the above two formulas  \eqref{eq1:rho} and \eqref{Uepsspace}.  Multiplying both sides of \eqref{eq1:rho} by $\int_0^\infty e^{-D(a)  -\eps^{\beta} sa } da$, we obtain, 
\begin{align*} 
\widehat \rho_\eps(s,x)& -\eps^\beta \text{ID}^{(1)}_\eps(s,x)  = \widehat U_\eps(s,x) \int_0^\infty e^{-D(a)  -\eps^{\beta} sa } da
\\
&= \int_{\R^d} \big[\rho_\eps(s,y)- \eps^\beta \text{ID}^{(1)}_\eps(s,y)\big] \big[1 -\eps^{\beta} s \int_0^\infty  e^{-D(a)  -\eps^{\beta} sa }da \big] \;   \omega_\eps (x,y) dy \\
& \qquad +\eps^\beta\text{ID}^{(2)}_\eps(s,x)\int_0^\infty e^{-D(a)  -\eps^{\beta} sa } da.
\end{align*}
%
which we may rewrite as 
\begin{align}\label{rhoepsspace} 
\int_{\R^n} \big[\widehat \rho_\eps(s,x)-\widehat \rho_\eps(s,y)+ \eps^\beta s &\widehat \rho_\eps(s,y)  \int_0^\infty e^{-D(a) -\eps^{\beta} sa }da \big]  \omega_\eps (x,y)dy \nonumber
\\
&=\text{Q}^{(2)}_\eps(s,x) + \text{Q}^{(1)}_\eps(s,x).
\end{align}
with
\begin{equation}\label{defQ} 
\begin{cases}
\text{Q}^{(2)}_\eps(s,x) = \eps^\beta \text{ID}^{(2)}_\eps(s,x) \int_0^\infty e^{-D(a) -\eps^{\beta} sa }da,
\\[5pt]
\text{Q}^{(1)}_\eps(s,x) =\eps^\beta \int_{\R^d}\left[ \text{ID}^{(1)}_\eps(s,x)-\text{ID}^{(1)}_\eps(s,y)\big[1 -\eps^{\beta} s \int_0^\infty  e^{-D(a)  -\eps^{\beta} sa }da \big] \right]   \omega_\eps (x,y)  dy.
\end{cases} \end{equation}

We now analyze the different terms arising in the identity~\eqref{rhoepsspace}, focusing on the case when $d(a)$ is given by Eq.~\eqref{as:hom}. We begin with an observation.

\paragraph{General calculations.}
We first recall that 
\begin{align}\label{INT1} 
\eps^\beta \int_0^\infty e^{-D(a) -\eps^{\beta} sa }da& = \eps^{ \beta} \int_0^\infty  \frac{e^{-\eps^{\beta} sa }}{(1+a)^\alpha} da =\eps^{\alpha \beta} \int_0^\infty \frac{e^{-st}}{(\eps^{{\beta}} + t)^\al}dt \notag\\
&= \eps^{\alpha \beta}  \int_0^\infty \frac{e^{-st}}{ t^\al}dt  \; [1+O(\eps^\beta) ] =  \eps^{\alpha \beta} {\cal L}_s (t^{-\alpha} )\; [1+O(\eps^\beta) ].
\end{align}

In a similar way, with the change of variable $a \to t= \eps^\beta (a-a')$
\begin{align}\label{INT2} 
\eps^\beta \int_{a'}^\infty &e^{D(a')-D(a) -\eps^{\beta} s(a-a') }da = \eps^{\al\beta} \int_0^\infty \frac{(1+a')^\alpha}{\big(t+(1+a')\eps^\beta \big)^\alpha}e^{-st}dt \notag
\\
&= \eps^{\alpha\beta} \int_0^\infty \f{(1+a')^\al}{t^\al}e^{-st}dt \; [1 +O(\eps^{ \beta})]
= \eps^{\alpha\beta} (1+a')^\al {\cal L}_s (t^{-\alpha} )\; [1 +O(\eps^{\beta})].
\end{align}

\paragraph{The terms for $\widehat\rho_\eps$.}

We begin with {\em the first two terms on the left hand side} (LHS) of \eqref{rhoepsspace}. With assumptions~\eqref{as:Subomega}, in the weak sense, by testing  against a smooth test function with compact support $\psi(x)$, they give 
\begin{align*} 
 \int_{\R^{2n}} \psi(x) & [ \widehat \rho_\eps(s,y) - \widehat \rho_\eps(s,x)] \omega_\eps (x,y)dy dx = 
\int_{\R^{2n}} \widehat \rho_\eps(s,x) [ \psi(y) -\psi(x)]  \omega_\eps (x,y)dy dx
\\
&= {\eps^2} \int_{\R^{2n}} \widehat \rho_\eps(s,x) [\p^2_{i,j} \psi(x) (y_i-x_i)(y_j-x_j) +O(|x-y|^3) ]\omega_\eps (x,y)dy dx
\\
&={\eps^2}  \int_{\R^{n}} \widehat \rho_\eps(s,x) [A(x):\nabla\nabla \psi(x)+O(\eps)] dx .\end{align*}

The {\em third term on the (LHS) of \eqref{rhoepsspace}},  can be treated with the same argument and yields
\begin{align} 
\eps^\beta s \int_{\R^n}   &\widehat \rho_\eps(s,y)  \int_0^\infty e^{-D(a) -\eps^{\beta} sa }da \omega_\eps (x,y)dy= \eps^\beta s \widehat \rho_\eps(s,x)\int_0^\infty e^{-D(a) -\eps^{\beta} sa }da [1 +O(\eps^2)]
\label{lhsrho}
\end{align}

Following the calculation for \eqref{FundLaplace2}, when  $d(a) = \frac{\alpha}{1+a}$, $0<\alpha <1$, 
the RHS of \eqref{lhsrho} can be written, with $\tau =\eps^{\beta} a $ and according to~\eqref{INT1} 
\begin{align*} 
\eps^{ \beta} s   \int_0^\infty \rho_\eps (x,t) e^{-st}dt\int_0^\infty  \frac{e^{-\eps^{\beta} sa }}{(1+a)^\alpha} da
=&s\int_0^\infty \rho_\eps e^{-st}dt\int_0^\infty \frac{e^{-s\tau}}{(1+\eps^{-{\beta}} \tau)^\al}d\tau\nonumber\\
=&\eps^{\al\beta}s\int_0^\infty \rho_\eps e^{-st}\int_t^\infty\frac{e^{-s(\tau-t)}}{(\eps^\beta+\tau-t)^\al}d\tau dt\nonumber\\
=&\eps^{\alpha\beta}s\int_0^\infty e^{-s\tau}\int_0^\tau \frac{\rho_\eps(t)}{(\eps^\beta+\tau-t)^\al}dt d\tau\nonumber\\
=&\eps^{\al\beta}\int_0^\infty e^{-st}\partial_t\int_0^t\frac{\rho_\eps(\tau)}{(\eps^\beta+t-\tau)^\al } d\tau dt.
\end{align*}
Here the last equality is due to integration by parts.

As a conclusion, we have obtained that, in the weak sense, the contribution of the left hand side of\eqref{rhoepsspace} is given by 
\begin{align} \label{LHS_final}
LHS= -\eps^2 \nabla \nabla [A(x) \widehat \rho_\eps (s,x)]  \; [1+ O(\eps)] + \eps^{\al\beta}\int_0^\infty e^{-st}\partial_t\int_0^t\frac{\rho_\eps(\tau)}{(t-\tau)^\al } d\tau dt [1+O(\eps^{\beta})].
\end{align}
At this stage, one can already observe that terms of orders  $\eps^2$ and $\eps^{\al\beta}$ should be compatible and thus the choice $2=\alpha \beta$.

\paragraph{The term $\text{Q}^{(1)}_\eps(s,x)$.}
 
\commentout{
Keeping in mind the expression \eqref{rhoepsspace}, we have to evaluate
$\eps^\beta Q^{(1)}_\eps$ where
\begin{align*}
& Q^{(1)}_\eps(s,x):= \text{ID}^{(1)}_\eps(s,x)- \int_0^\infty \!\!\int_{\R^d} d(a) e^{-D(a)  -\eps^{\beta} sa } \text{ID}^{(1)}_\eps(s,y) \omega_\eps (x,y)dady,
\end{align*}
which we write as
\begin{align*}
 Q^{(1)}_\eps=& \text{ID}^{(1)}_\eps(s,x)- \int_0^\infty \!\!\int_{\R^d} d(a) e^{-D(a)  -\eps^{\beta} sa } \big(\text{ID}^{(1)}_\eps(s,y)-\text{ID}^{(1)}_\eps(s,x)\big) \omega_\eps (x,y)dady\nonumber\\
&-\int_0^\infty(d(a)+\eps^\beta s) e^{-D(a)  -\eps^{\beta} sa }da \text{ID}^{(1)}_\eps(s,x)+\eps^\beta s\int^\infty_0 e^{-D(a)  -\eps^{\beta} sa }da \text{ID}^{(1)}_\eps(s,x). \nonumber
\end{align*}
The first and third terms cancel and we can write
\begin{align*}
Q^{(1)}_\eps=&\eps^\beta s\int_0^\infty e^{-D(a)-\eps^\beta sa}da \text{ID}^{(1)}_\eps(s,x)+\eps^2 \int_{\R^d}  \big(\text{ID}^{(1)}_\eps(s,y)-\text{ID}^{(1)}_\eps(s,x)\big) \f{\omega_\eps (x,y)}{\eps^2}dy\nonumber\\
&+\eps^{\beta+2} s\int_0^\infty \!\!\int_{\R^d} e^{-D(a)  -\eps^{\beta} sa } \big(\text{ID}^{(1)}_\eps(s,y)-\text{ID}^{(1)}_\eps(s,x)\big) \f{\omega_\eps (x,y)}{\eps^2}dady\nonumber
\\
=&\eps^{\alpha\beta}s^\alpha \Gamma_0\text{ID}^{(1)}_\eps(s,x)+\eps^2 \int_{\R^d}  \big(\text{ID}^{(1)}_\eps(s,y)-\text{ID}^{(1)}_\eps(s,x)\big) \f{\omega_\eps (x,y)}{\eps^2}dy+O(\text{HOT}).
\end{align*}
Thanks to \eqref{eq:ID1}, the first term on the RHS can be further reformulated as 
}
According to the analysis for $\rho_\eps$, the jump kernel $\omega_\eps(x,y)$ produces an $O(\eps^2)$ compared to identity. Therefore, still in the weak sense, we have

\begin{align*}
\text{Q}^{(1)}_\eps(s,x)&= \eps^\beta\left[ \text{ID}^{(1)}_\eps(s,x)-\text{ID}^{(1)}_\eps(s,x) \; [1+O(\eps^2)] \; \big[1 -\eps^{\beta} s \int_0^\infty  e^{-D(a)  -\eps^{\beta} sa }da \big]\right]
\\
&= O(\eps^{2+\beta}) \; \text{ID}^{(1)}_\eps(s,x) \;  [1+O(\eps^{\alpha \beta})]
\end{align*}
according to \eqref{INT1}. The term $\text{ID}^{(1)}_\eps$ can be also evaluated as
\begin{align*}
\eps^\beta \text{ID}^{(1)}_\eps(s,x)&= \eps^\beta \int_0^\infty \!\!\int_0^\infty u^0(a',x)\frac{(1+a')^\alpha}{(t+1+a')^\alpha}e^{-\eps^\beta st}dtda' =O(\eps^{\alpha\beta})
\end{align*}
according to \eqref{INT2}.

As a consequence, the contribution of the term $\text{Q}^{(1)}_\eps(s,x)$, vanishes compared to those arising in~\eqref{LHS_final}.

\commentout{Finally, thanks to \eqref{FundLaplace}, we find
\begin{align} \label{ID1Final}
s^\alpha \Gamma_0 \text{ID}^{(1)}_\eps(s,x)
=\eps^{\al\beta-\beta}\int_0^\infty\partial_t\Big(\int_0^t\frac{\int_0^\infty u^0(a',x)\f{(1+a')^\al}{(t+\eps^\beta(1+a'))^\al}da'}{(t-\tau)^\al}d\tau\Big) e^{-st}dt.
\end{align}

Finally, we get
\begin{align} \label{ID1total}
\eps^\beta Q^{(1)}_\eps=\eps^{\al\beta-\beta}\int_0^\infty\partial_t\Big(\int_0^t\frac{\int_0^\infty u^0(a',x)\f{(1+a')^\al}{t^\al}da'}{(t-\tau)^\al}d\tau\Big) e^{-st}dt + O(HOT)
\end{align}
}

\paragraph{The term $\text{Q}^{(2)}_\eps(s,x)$.}
To analyze the limit of Eq.~\eqref{rhoepsspace}, we finally need to analyze the term $\text{Q}^{(2)}_\eps(s,x)$ in~\eqref{defQ}.

Using \eqref{INT1}, we have
\[
\text{Q}^{(2)}_\eps(s,x) = e^{\alpha \beta} \int_0^\infty \frac{e^{ -st}}{t^\alpha}dt [1+O(\eps^\beta)] \; \text{ID}^{(2)}_\eps(s,x).
\]
Also, according to the analysis for $\rho_\eps$, in the definition \eqref{IDeps} of $\text{ID}^{(2)}_\eps$, the jump kernel $\omega_\eps(x,y)$ produces an $O(\eps^2)$ compared to identity. Therefore, we also have
\begin{align*}
\text{ID}^{(2)}_\eps(s,x)& = \int_{a'=0}^\infty u^0(a',x) \; [1 +O(\eps^2)] 
\big[1- \eps^\beta s \int_{a'}^\infty e^{-D(a) +D(a') -\eps^{\beta} s(a-a') }  \big]   da da'
 \\
 &= \ 
 \int_{0}^\infty u^0(a,x) (1+a)^\al [1 +O(\eps^2)][1 +O(\eps^{\alpha \beta})] da,
\end{align*}
where we have used the formula \eqref{INT2}.

We conclude that
\begin{align}\label{ID2_final}
\text{Q}^{(2)}_\eps(s,x)& = e^{\alpha \beta} \int_0^\infty \frac{e^{ -st}}{t^\alpha}dt
\int_{0}^\infty u^0(a,x) (1+a)^\al da \; [1 +O(\eps^2+\eps^{\alpha \beta})].
\end{align}

\paragraph{Conclusion (sub-diffusion case).}
Back to Eq.~\eqref{rhoepsspace}, we choose $\alpha \beta=2$.  Putting together \eqref{LHS_final} and \eqref{ID2_final},  we find
\begin{align*}\
\int_0^\infty e^{-st}\partial_t \int_0^t & \frac{\rho_\eps(\tau,x)}{(
t-\tau)^\al }d\tau dt  \nonumber\\
=& \nabla \nabla [A(x) \widehat \rho_\eps(s,x)]+
\int_0^\infty t^{-\al}e^{-st}\int_0^{\infty} u^0(a,x) da dt +O(\eps).
\end{align*}
This  is the Laplace transform of the equation
\begin{equation} \label{subdifins}
\partial_t\int_0^t\frac{\rho_\eps(\tau,x)}{(t-\tau)^\al } d\tau  =  \nabla \nabla \big(A(x)  \rho_\eps (s,x)\big) 
+ t^{-\al}\int_0^{\infty}u^0(a,x)da +O(\eps).
\end{equation}

As $\eps$ vanishes, we find that the limit $\rho_0$ of $\rho_\eps$  (which exists by uniqueness), satisfies Eq.~\eqref{AbNormaDiff}.

\paragraph{The diffusion case.}  When $d(a):=D_0$, where $D_0$ is a positive constant, then  $D(a) = D_0 a$. We can compute explicitly
\[
\int_0^\infty  e^{-D_0a -\eps^{\beta} sa } da = \frac{ 1}{D_0+\eps^{\beta} s},
\qquad
\int_{a'}^\infty e^{-D(a)+D(a')-\eps^\beta s (a-a')}da = \frac{ 1}{D_0+\eps^{\beta} s},
\]
when $d(a)=\frac{\al}{1+a}$ with $\al>1$, similar formula hold with $e^{-D(a)}=\frac{1}{(1+a)^\al}$ and $D_0=1-\al$ in the above equalities, at leading order.
Then the expressions in~\eqref{rhoepsspace} are simpler and reads  at higher order
\[
\eps^2 \nabla \nabla \big( A(x) \widehat \rho_\eps(s,x)\big) +\eps^\beta  \frac{s}{D_0}   \widehat \rho_\eps(s,x)=  \text{Q}^{(1)}_\eps(s,x)
+\text{Q}^{(2)}_\eps(s,x),
\]
from which we conclude that the scale is now $\beta =2$.

The terms on the right hand side can be evaluated easily using their definitions in ~\eqref{defQ}. For $\text{Q}^{(1)}_\eps$, we find 
\[
\text{Q}^{(1)}_\eps(s,x)=\eps^\beta O(\eps^2) \;  \text{ID}^{(1)}_\eps(s,x)
\]
and, using the definition in \eqref{eq:ID1}, at leading order
\[
\text{ID}^{(1)}_\eps(s,x)=\frac{1}{D_0}  \int_0^\infty u^0(a',x) .
\]
Therefore the contribution of $\text{Q}^{(1)}_\eps(s,x)$ is again at smaller scale.

The term $\text{Q}^{(2)}_\eps$ yields
\[
\text{Q}^{(2)}_\eps(s,x)= \frac{\eps^\beta}{D_0} \text{ID}^{(2)}_\eps(s,x)= \frac{\eps^\beta}{D_0} \int_0^\infty u^0(a,x) da \; [1+O(\eps^\beta)].
\]

Finally, we find, in the weak sense
\[
\eps^2 \nabla \nabla \big( A(x) \widehat \rho_\eps(s,x)\big) +\eps^\beta  \frac{s}{D_0}   \widehat \rho_\eps(s,x)=  
\frac{\eps^\beta}{D_0} \int_0^\infty u^0(a,x) da  +O(\eps^{2\beta}),
\]
which is the Laplace transform of the equation
\[
\partial_t \rho_\eps(t,x) =D_0 \nabla \nabla  \big(A(x) \rho_\eps(t,x) \big) + \delta(t) \rho^0(x)+O(\eps^2)
\]
Here we recall our convention that $u(t,a,x)$ and $\rho(t,x)$ are extended by $0$ for $t<0$ so that the Dirac mass represents the initial data as a jump. Consequently, the quantity $\rho^0(x)=\int_0^\infty u^0(x,a) da$ represents the initial condition for $\rho_0(t,x)$.
\end{proof}

\subsection{Remarks on regularity}
\label{sec:tDecay}

\paragraph{Regularity of $\rho_0$ in the sub-diffusion case.} 
When $\al <1$, the only uniform bound available for $\rho_\eps$ is that $\rho_\eps \geq 0$ and $\int_{\R^d} \rho_\eps(t,x)dx =M$. Consequently the limit $\rho_0$ is a measure. However it is standard to obtain some Sobolev estimates in $L^2$ (see \cite{DongKim_2020} for further regularity). We consider the case when $A=Id$, the identity matrix, to simplify and explain now how to find the following $H^1$ bound for $\rho_0$:
\begin{align} \label{reg:v}
 \frac 12   \partial_t \int_0^t \!  \!  \int_{\R^d} \frac{v^2(\tau,x)}{(t-\tau)^\al} d\tau dx& +\int_{\R^d} |\nabla v |^2 dx
 +\frac{\al}{2} \int_0^t\! \!  \int_{\R^d} \frac{|v(\tau,x)-v(t,x)|^2}{(t-\tau)^{\al+1}} d\tau dx \notag+\int_{\R^d}\frac{v^2(t)}{2t^\alpha}dx
 \\
 &=-\int_{\R^d} \nabla v(t,x)\cdot\nabla \rho^0(x) dx,
\end{align}
where $v(t,x) =\rho_0(t,x)-\rho^0(x)$. 

Notice that $v$ satisfies the sub-diffusion equation
\[
  \partial_t \int_0^t \frac{v(\tau)}{(t-\tau)^\al} d\tau - \Delta v(t,x) = \Delta \rho^0(x).
\]
From Remark \ref{remark5},  $v(0,x)=0$. Then, \eqref{reg:v} is obtained by multiplying this equation by $v(t,x)$ and using integration by parts. The fractional time derivative term  follows from the following identity: for $v\in C^1([0,\infty))$ satisfying $v(0)=0$, one has
\begin{align} \label{Chainrule}
\frac 12 \partial_t \int_0^t \frac{v^2(\tau)}{(t-\tau)^\al} d\tau + \frac{v^2(t)}{2t^\al} + \frac \al 2 \int_0^t \frac{|v(\tau)-v(t)|^2}{(t-\tau)^{\al+1}} d\tau=
v(t)  \partial_t \int_0^t \frac{v(\tau)}{(t-\tau)^\al} d\tau.
\end{align}

To show this, on the one hand, we write 
\begin{align} \label{reg1}
\partial_t \int_0^t \frac{v(\tau)}{(t-\tau)^\al} d\tau &= \partial_t \int_0^t \frac{v(t-s)}{s^\al} ds=- \int_0^t  \frac{\partial_s [v(t-s)-v(t)]}{s^\al} ds \nonumber\\
&=-\frac{v(t-s)-v(t)}{s^\al}\Big|_{s=0}^{s=t}-\al\int_0^t\frac{v(t-s)-v(t)}{s^{\al+1}}ds\nonumber\\
&=\frac{v(t)}{t^\al}+\lim_{s\to 0}\frac{v(t-s)-v(t)}{s}s^{1-\al}-\alpha \int_0^t  \frac{ v(t-s)-v(t)}{s^{\al+1}} ds\nonumber\\
&=\frac{v(t)}{t^\al}-\alpha \int_0^t  \frac{ v(t-s)-v(t)}{s^{\al+1}} ds.
\end{align}

On the other hand, applying the result to $v^2(\tau)$ we write
\begin{align*}
\partial_t \int_0^t \frac{v(\tau)^2}{(t-\tau)^\al} d\tau &= \frac{v(t)^2}{t^\al}-\alpha \int_0^t  \frac{ v(t-s)^2-v(t)^2}{s^{\al+1}} ds
\\
&=\frac{v(t)^2}{t^\al} -\alpha \int_0^t  \frac{ |v(t-s)-v(t)|^2}{s^{\al+1}} ds - 2 v(t) \alpha \int_0^t  \frac{ v(t-s)-v(t)}{s^{\al+1}} ds.
\end{align*}
Thanks to \eqref{reg1}, we conclude
\begin{align*}
\partial_t \int_0^t \frac{v^2(\tau)}{(t-\tau)^\al} d\tau &=
\frac{v^2(t)}{t^\al} -\alpha \int_0^t  \frac{ |v(t-s)-v(t)|^2}{s^{\al+1}} ds + 2 v(t) \left[\partial_t \int_0^t \frac{v(\tau)}{(t-\tau)^\al} d\tau -\frac{v(t)}{t^\al} \right],
\end{align*}
which gives the announced result Eq.~\eqref{Chainrule}.

As in \cite{Allen2017}, notice that it also holds, for all convex function $\Phi$ such that $\Phi(0)=0$, 
\[
\partial_t \int_0^t \frac{\Phi(v(\tau))}{(t-\tau)^\al} d\tau \leq \Phi'(v(t))\partial_t \int_0^t \frac{v(\tau)}{(t-\tau)^\al} d\tau.
\]
It would be interesting to understand if a related estimate holds uniformly in $\eps$.
%

\paragraph{Time regularity.}

The time regularity for $U(t,x)$ can also be obtained with the following observation.
Let $z_\eps= \partial_t u_\eps$. Differentiating in time on both sides of Eq.~\eqref{eq:AgeSpace}, we find that $z_\eps$ satisfies the same equation. Therefore, we also obtain
\[
|\partial_t z_\eps (t,a,x)| \leq  \frac{C^1}{(1+a)^\alpha}
\]
under the assumption \(|\partial_t z_\eps (0,a,x)| \leq  \frac{C^1}{(1+a)^\alpha;}\) which, using Eq.~\eqref{eq:AgeSpace}, requires that
\begin{equation}
\partial_a u^0_\eps + d(a) u^0_\eps = O(\eps^\beta).
\end{equation} 
This 'near equilibrium' condition implies, after integration in $y$ and $a$ against the weight $\frac{\omega_\eps(x,y)}{1+a}$, that the renewal term  $U_\eps (t,x)$ satisfies
\[
|\partial_t U_\eps (t,x)| \leq \frac{C^1}{\alpha}.
\]

%
%
%


\end{document}